\documentclass[12pt]{article}
\usepackage[pagewise]{lineno}\nolinenumbers
\usepackage[english]{babel}
\usepackage{amsmath}
\usepackage{amssymb}
\usepackage[mathscr]{eucal}
\usepackage{graphicx}
\usepackage{centernot}
\usepackage{color}
\usepackage{bm}
\usepackage{mathrsfs}
\usepackage{amsthm}
\usepackage{enumerate}
\usepackage{eqlist}
\allowdisplaybreaks
\numberwithin{equation}{section}
\newtheorem{Theorem}{\sc Theorem}
\newtheorem{Definition}[Theorem]{\sc Definition}
\newtheorem{Proposition}[Theorem]{\sc Proposition}
\newtheorem{Lemma}[Theorem]{\sc Lemma}
\newtheorem{Corollary}[Theorem]{\sc Corollary}

\newcommand{\eps}{\varepsilon}

\def\sqr#1#2{{
        \vcenter{
            \vbox{\hrule height.#2pt
                \hbox{\vrule width.#2pt height#1pt \kern#1pt
                    \vrule width.#2pt
                }
                \hrule height.#2pt
            }
        }
}}

\def\div{\mathop{\rm div}\nolimits}

\def\R{\mathbb{R}}
\def\N{\mathbb{N}}

\def\dx{{\rm d}x}
\def\dy{{\rm d}y}
\def\lista#1
{{ \itemindent 0.0cm \labelsep .2cm \leftmargin 0.8cm \rightmargin
        0.0cm \labelwidth 0.6cm \topsep 0.0mm
        \parsep 0.0mm
        \itemsep 0.0mm
        \begin{list}{}
            { \setlength{\leftmargin}{.8cm} \setlength{\rightmargin}{0.0cm}
                \setlength{\parsep}{0.0mm} \setlength{\topsep}{.0mm}
                \setlength{\parskip}{.0cm} \setlength{\itemsep}{.0cm} }
            {#1}\end{list}} }

\begin{document}
\title{ Differential inclusion systems with fractional competing operator and multivalued fractional convection term}
\author{
Jinxia Cen\footnote{School of Mathematical Sciences, Zhejiang Normal University, Jinhua 321004, P.R. China. E-mail: jinxcen@163.com},\; 
Salvatore A. Marano\footnote{Dipartimento di Matematica e Informatica, Università degli Studi di Catania, Viale A. Doria 6, 95125 Catania, Italy. E-mail: marano@dmi.unict.it},\; 
Shengda Zeng\footnote{Corresponding author. National Center for Applied Mathematics in Chongqing and School of Mathematical Sciences, Chongqing Normal University, Chongqing 401331, P.R. China. E-mail: zengshengda@163.com}
}
%
\date{}
\maketitle
\begin{abstract}
In this work, the existence of solutions (in a suitable sense) to a family of inclusion systems involving fractional, possibly competing, elliptic operators, fractional convection, and homogeneous Dirichlet boundary conditions is established. The technical approach exploits Galerkin's method and a surjective results for multifunctions in finite dimensional spaces as well as approximating techniques.    
\end{abstract}
\section{Introduction}
%
Let $\Omega$ be a bounded domain in $\R^N$, $N\geq 3$, with a smooth boundary $\partial \Omega$, let $\mu_1,\mu_2\in\R$, and let $F_1,F_2:\Omega\times\R^2\times\R^{2N}\to 2^\R$ be two compact convex-valued multifunctions. Consider the differential inclusion system
\begin{equation}\label{prob}
\left\{\begin{array}{lll}
(-\Delta)_{p_1}^{s_1}u_1+\mu_1 (-\Delta)_{q_1}^{t_1}u_1 \in F_1(x,u_1,u_2,
D^{r_1} u_1,D^{r_2} u_2) & \mbox{ in $\Omega$},\\
(-\Delta)_{p_2}^{s_2}u_2+\mu_2 (-\Delta)_{q_2}^{t_2}u_2 \in F_2(x,u_1,u_2,
D^{r_1} u_1,D^{r_2} u_2) & \mbox{ in $\Omega$},\\
u_1=u_2=0 & \mbox{ in $\mathbb R^N\backslash\Omega$},
\end{array}\right.
\end{equation}
where
\begin{itemize}
\item[$({\rm H}_1)$] $0<t_i<r_i<s_i\le 1$ and $1<q_i<p_i<\frac{N}{s_i}$ for each $i=1,2$.
\end{itemize}
The symbol $(-\Delta)_p^s$, with $p>1$ and $0<s<1$, denotes the fractional $p$-Laplacian, defined by setting, provided $u$ is smooth enough,
$$(-\Delta)^s_p u(x):=2\lim_{\varepsilon\to 0^+}\int_{\R^N\setminus B_\varepsilon(x)}
\frac{|u(x)-u(y)|^{p-2}(u(x)-u(y))}{|x-y|^{N+ps}}\,\dy,\quad x\in\R^N$$
with $B_\varepsilon(x):=\{z\in \mathbb R^N\,\mid\,\|z-x\|_{\mathbb R^N}<\varepsilon\}$.
When $s=1$ it becomes the classical $p$-Laplacian, namely
$$-\Delta_p u:=-\div(|\nabla u|^{p-2}\nabla u).$$
Moreover, $D^s u$ indicates the \textit{distributional Riesz fractional gradient} of $u$ in the sense of \cite{SS1,SS2}. If $u$ appropriately decays and is sufficiently smooth then, setting
$$c_{N,s}:=-\frac{2^{s}\,\Gamma\left(\frac{N+s+1}{2}\right)}
{\pi^\frac{N}{2}\,\Gamma\left(\frac{1-s}{2}\right)}\, ,$$ 
one has \cite[pp. 289 and 298]{SS2} 
\begin{equation*}
D^s u(x):=c_{N,s}\lim_{\varepsilon\to 0^+}\int_{\R^N\setminus B_\varepsilon(x)}
\frac{u(x)-u(y)}{|x-y|^{N+s}}\,\frac{x-y}{|x-y|}\dy,\quad x\in\R^N.
\end{equation*}
The right-hand sides $F_1$ and $F_2$ satisfy the conditions below, where, to avoid cumbersome formulae, we shall write
\begin{equation}\label{defpstar}
y:=(y_1,y_2),\quad z:=(z_1,z_2),\quad p_i^*:=\frac{Np_i}{N-s_ip_i},\quad i=1,2.    
\end{equation}
\begin{itemize}
\item[$({\rm H}_2)$] $x\mapsto F_i(x,y,z)$ is measurable on $\Omega$ for all $(y,z)\in\R^2\times\R^{2N}$ and $(y,z)\mapsto F_i(x,y,z)$ is upper semi-continuous for a.e. $x\in\Omega$. 
\item[$({\rm H}_3)$] There exist $m_i>0$, $\delta_i\in L^{(p_i^*)'}(\Omega)$, $i=1,2$, such that
\begin{align*}
\sup_{w_i\in F_i(x,y,z)}|w_i|
\le & m_1 \left(|y_1|^\frac{p_1^*}{(p_i^*)'}+|y_2|^\frac{p_2^*}{(p_i^*)'}
+|z_1|^\frac{p_1}{(p_i^*)'}+|z_2|^\frac{p_2}{(p_i^*)'}\right)+\delta_i(x)
%
\end{align*}
a.e. in $\Omega$ and for all $(y,z)\in\R^2\times\R^{2N}$.
\item[$({\rm H}_4)$] There are $M_i,M_i'>0$, $\sigma_i\in L^1(\Omega)_+$, $i=1,2$, fulfilling
\begin{align*}
& w_i y_i\le M_i(|y_1|^{p_1}+|y_2|^{p_2})+M_i'(|z_1|^{p_1}+|z_2|^{p_2})+\sigma_i(x)
%
\end{align*} 
a.e. in $\Omega$ and for all $(y,z)\in\R^2\times\R^{2N}$, $w_i\in F_i(x,y,z)$.
\end{itemize}
The involved differential operators are of the type 
\begin{equation*}
A_\mu(u):= (-\Delta)_{p}^{s}u+\mu (-\Delta)_{q}^{t}u,\quad u\in W^{s,p}_0(\Omega),    
\end{equation*}
where $\mu\in\R$, $0<t\leq r\leq s\le 1$, $1<q<p<\frac{N}{s}$, while convection comes from the presence of fractional gradients $D^r u$ at right-hand sides. $A_\mu$ exhibits different behaviors depending on the values of $t,s\in (0,1]$. Precisely, if $t=1$, then $t=r=s=1$, Problem \eqref{prob} falls inside the local framework, and  has been already investigated in some recent works; see, e.g., \cite{GMM-JMAA} for single-valued reactions and \cite{Mo1,CMZ-AML} as regards multi-valued ones. Moreover, the nature of $A_\mu$ drastically changes depending on $\mu$. In fact, when $\mu>0$, the operator $A_{\mu}$ is basically patterned after the (possibly) fractional $(p,q)$-Laplacian, which turns out non-homogeneous because $p\neq q$. If $\mu=0$ it coincides with the fractional $p$-Laplacian. Both cases have been widely investigated and meaningful results are by now available in the literature. On the contrary, for $\mu<0$ the operator $A_{\mu}$ contains the \textit{difference} between the fractional $p$- and $q$-Laplacians. It is usually called competitive and, as already pointed out in \cite{LLMZ,Mo}, doesn't comply with any ellipticity or monotonicity condition. In fact, given $u_0\in W^{s,p}_0(\Omega)\setminus\{0\}$ and chosen $u:=\tau u_0$, $\tau>0$, the expression
$$\langle A_{\mu}(u),u\rangle=\tau^p\Vert\nabla u_0\Vert_{s,p}^p
+\mu \tau^q\Vert u_0\Vert_{s,q}^q$$
turns out negative for $\tau$ small and positive when $\tau$ is large. Therefore, nonlinear regularity theory, comparison principles, as well as existence theorems for pseudo-monotone maps cannot be employed. Moreover, since the reactions are multi-valued and contain the fractional gradient of the solutions, also variational techniques are no longer directly usable. To overcome these difficulties we first exploit Galerkin's method, thus working with a sequence $\{E_n\}$ of finite dimensional functional spaces. For each $n\in\N$, an \textit{approximate solution} $(u_{1,n},u_{2,n})\in E_n$ to \eqref{prob} is obtained via a suitable version (cf. Proposition \ref{surthm}) of a classical surjectivity result. Next, letting $n\to+\infty$ yields a solution in a generalized sense (cf. Definitions \ref{Def1} and \ref{Def2}), which turns out weak sense once $\mu_1\wedge\mu_2\geq 0$.

Fractional gradients were first introduced more than sixty years ago by Horváth \cite{H}, but took a great interest especially after the works of Shieh and Spector \cite{SS1,SS2,CS2}. The operator $D^s u$ appears as a natural non-local version of $\nabla u$, to whom it formally tends when $s\to 1^-$. It enjoys good geometric and physical properties \cite{BC,SI}, like invariance under translations or rotations, homogeneity of order $s$, continuity, etc.

Section 2 contains some auxiliary results and the functional framework needed for handling both fractional gradients and the fractional $p$-Laplacian. The existence of (generalized, strong generalized, or weak) solutions to \eqref{prob} is established in Section 3.
\section{Preliminaries}\label{prelim}
Let $X,Y$ be two nonempty sets. A multifunction $\Phi:X\to 2^Y$ is a map from $X$ into the family of all nonempty subsets of $Y$. A function $\varphi:X\to Y$ is called a selection of $\Phi$ when $\varphi(x)\in\Phi(x)$ for every $x\in X$. Given $B\subseteq Y$, put 
$\Phi^-(B):=\{ x\in X\,\mid\,\Phi(x)\cap B\neq\emptyset\}.$
If $X,Y$ are topological spaces and $\Phi^-(B)$ turns out closed in $X$ for all closed sets $B\subseteq Y$ then we say that $\Phi$ is upper semi-continuous. Suppose $(X,\mathcal{F})$ is a measurable space and $Y$ is a topological space. The multifunction $\Phi$ is called measurable when $\Phi^-(B)\in\mathcal{F}$ for every open set $B\subseteq Y$. The result below, stated in \cite[p. 215]{ADZ}, will be repeatedly useful.
\begin{Proposition}\label{supresult}
Let $F:\Omega\times\R^h\to 2^\R$ be a closed-valued multifunction such that:
\begin{itemize}
\item $x\mapsto F(x,\xi)$ is measurable for all $\xi\in\R^h$;
\item $\xi\mapsto F(x,\xi)$ is upper semi-continuous for a.e. $x\in\Omega$.
\end{itemize}
Let $w:\Omega\to\R^h$ be measurable. Then the multifunction $x\mapsto F(x,w(x))$ admits a measurable selection.
\end{Proposition}
Let $(X,\Vert\cdot\Vert)$ be a real normed space with topological dual $X^*$ and duality brackets $\langle\cdot,\cdot\rangle$. Given a nonempty set $A\subseteq X$, define $|A|:=\sup_{x\in A}\Vert x\Vert$. We say that $\varphi:X\to X^*$ is \emph{monotone} when 
$$\langle\varphi(x)-\varphi(z),x-z\rangle\geq 0\quad\forall\, x,z\in X,$$
and \emph{of type $(\mathrm{S})_+$} provided
\begin{equation*}
x_n\rightharpoonup x\;\;\mbox{in $X$,}\;\;\limsup_{n\to+\infty}\langle\varphi(x_n),x_n-x\rangle \le 0\implies x_n\to x\;\;\mbox{in $X$.}   
\end{equation*}
The next elementary result \cite[Proposition 2.1]{GM-ANONA} ensures that condition $(\mathrm{S})_+$ holds true for the fractional $(p,q)$-Laplacian.
\begin{Proposition}\label{sumop}
Let $\varphi:X\to X^*$ be of type $(\mathrm{S})_+$ and let $\psi:X\to X^*$ be monotone. Then $\varphi+\psi$ satisfies condition $(\mathrm{S})_+$. 
\end{Proposition}
A multifunction $\Phi:X\to 2^{X^*}$ is called coercive provided
$$\lim_{\Vert x\Vert\to\infty}\frac{\inf\{\langle x^*,x\rangle\,\mid\, x\in X,\, x^*\in\Phi(x)\}}{\Vert x\Vert}=+\infty\, .$$
The following result is a direct consequence of~\cite[Proposition 3.2.33]{GP1}. 
\begin{Theorem}\label{surthm}
Let $X$ be a finite-dimensional normed space and let $\Phi:X\to 2^{X^*}$ be a convex compact-valued multifunction. Suppose $\Phi$ is upper semi-continuous and coercive. Then there exists $\hat{x}\in X$ satisfying $0\in\Phi(\hat{x})$. 
\end{Theorem}
Hereafter, if $X$ and $Y$ are two topological spaces, the symbol $X\hookrightarrow Y$ means that $X$ continuously embeds in $Y$. Given $p>1$, put $p':=\frac{p}{p-1}$, denote by $\Vert\cdot\Vert_p$ the usual norm of $L^p(\Omega)$, and  indicate with $\Vert\cdot\Vert_{1,p}$ the norm on $W^{1,p}_0(\Omega)$ arising from Poincaré's inequality, namely
\begin{equation*}
\Vert u\Vert_{1,p}:=\Vert\nabla u\Vert_p\, ,\quad u\in W^{1,p}_0(\Omega).
\end{equation*}
%
%
If $u\in W^{1,p}_0(\Omega)$, we set $u(x)=0$ on $\R^N\setminus\Omega$; cf. \cite[Section 5]{DNPV}. Fix $s\in(0,1)$. The Gagliardo seminorm of a measurable function $u:\R^N\to\R$ is
\begin{equation*}
[u]_{s,p}:=\left(\int_{R^N\times\R^N}\frac{|u(x)-u(y)|^p}{|x-y|^{N+ps}}{\dx}{\dy} \right)^\frac{1}{p},
\end{equation*}
while $W^{s,p}(\R^N)$ denotes the fractional Sobolev space
\begin{equation*}
W^{s,p}(\R^N):= \left\{u\in L^p(\R^N):\ [u]_{s,p}<+\infty \right\},
\end{equation*}
endowed with the norm
\begin{equation*}
\|u\|_{W^{s,p}(\mathbb{R}^N)}:=\left(\|u\|^p_{L^p(\mathbb{R}^N)}+[u]_{s,p}^p\right)^\frac{1}{p}.
\end{equation*}
As usual, on the space
\begin{equation*}
W^{s,p}_0(\Omega):=\{u\in W^{s,p}(\R^N):u=0\;\mbox{a.e. in}\;\R^N\setminus\Omega\}
\end{equation*}
we will consider the equivalent norm
$$\| u\|_{s,p}:=[u]_{s,p},\quad u\in W^{s,p}_0(\Omega).$$
Let $W^{-s,p'}(\Omega):=(W^{s,p}_0(\Omega))^*$ and let $p^*_s$ be the fractional Sobolev critical exponent, i.e., $p^*_s=\frac{Np}{N-sp}$ when $sp<N$, $p^*_s= +\infty$ otherwise. Thanks to Propositions 2.1--2.2, Theorem 6.7, and Corollary 7.2 of \cite{DNPV} one has
\begin{Proposition}\label{fracemb}
If $1\leq p<+\infty$ then:
\begin{itemize}
\item[{\rm (a)}] $0<s'\le s''\le 1\;\implies\; W^{s'',p}_0(\Omega)\hookrightarrow W^{s',p}_0(\Omega)$. 
\item[{\rm (b)}] $W^{s,p}_0(\Omega)\hookrightarrow L^r(\Omega)$ for all $r\in [1, p^*_s]$.
\item[{\rm (c)}] The embedding in {\rm (b)} is compact once $r<p^*_s<+\infty$.
\end{itemize}
\end{Proposition}
However, contrary to the non-fractional case, we know \cite{MS} that
$$1\leq q<p\leq+\infty\;\;\centernot\implies W^{s,p}_0(\Omega) \subseteq W^{s,q}_0(\Omega).$$
Define, for every $u,v\in W^{s,p}_0(\Omega)$,
\begin{equation*}
\langle(-\Delta)^s_p u,v\rangle:=
\int_{\R^N\times\R^N}\frac{|u(x)-u(y)|^{p-2}(u(x)-u(y))(v(x)-v(y))}{|x-y|^{N+ps}}\dx\dy\, .
\end{equation*}
The operator $(-\Delta)_p^s$ is called (negative) $s$-fractional $p$-Laplacian. It possesses the following properties.
\begin{itemize}
\item[$({\rm p}_1)$] $(-\Delta)^s_p:W^{s,p}_0(\Omega)\rightarrow W^{-s,p'}(\Omega)$ turns out monotone, continuous, and of type $({\rm S})_+$; vide, e.g., \cite[Lemma 2.1]{FI}.
\item[$({\rm p}_2)$] One has
$$\Vert(-\Delta)^s_p u\Vert_{W^{-s,p'}(\Omega)}\leq\Vert u\Vert_{s,p}^{p-1} 
\quad\forall\, u\in W^{s,p}_0(\Omega).$$
Hence, $(-\Delta)^s_p$ maps bounded sets into bounded sets.
\item[$({\rm p}_3)$] The first eigenvalue $\lambda_{1,p,s}$ of $(-\Delta)^s_p$ is given by (cf. \cite{LL})
\begin{equation*}
\lambda_{1,p,s}=\inf_{u \in W_0^{s,p}(\Omega),u \neq 0}
\frac{\|u\|_{s,p}^p}{\|u\|_p^p}\, .
\end{equation*}
\end{itemize}
To deal with distributional fractional gradients, we first introduce the Bessel potential spaces $L^{\alpha,p}(\R^N)$, where $\alpha>0$. Set, for every $x\in\R^N$,
\begin{equation*}
g_\alpha(x):= \frac{1}{(4 \pi)^\frac{\alpha}{2}\Gamma\left(\frac{\alpha}{2}\right)} \int_{0}^{+\infty}e^{\frac{-\pi|x|^2}{\delta}} e^{\frac{-\delta}{4\pi}}\delta^{\frac{\alpha-N}{2}}\frac{{\rm d}\delta}{\delta}\, .  
\end{equation*}
On account of \cite[Section 7.1]{MI} one can assert that:
\begin{itemize}
\item[1)] $g_\alpha\in L^1(\R^N)$ and $\|g_\alpha\|_{L^1(\R^N)}=1$.
\item[2)] $g_\alpha$ enjoys the semi-group property, i.e., $g_\alpha \ast g_\beta= g_{\alpha+\beta}$ for any $\alpha,\beta>0$, with $*$ being the convolution operator.
\end{itemize}
Now, put
\begin{equation*}
L^{\alpha,p}(\R^N):=\{u:\, u= g_\alpha\ast\tilde u\;\mbox{for some}\;\tilde u\in L^p(\R^N)\}
\end{equation*}
as well as
$$\|u\|_{L^{\alpha,p}(\R^N)}= \|\tilde u\|_{ L^p(\R^N)}\;\;\mbox{whenever}\;\; u=g_\alpha\ast\tilde u.$$ 
Using 1) and 2) we easily get
$$0<\alpha<\beta\;\implies\; L^{\beta,p}(\R^N)\subseteq L^{\alpha,p}(\R^N)\subseteq L^p(\R^N).$$
Moreover, by \cite[Theorem 2.2]{SS1}, one has 
\begin{Theorem}\label{besselspace} 
If $1<p<+\infty$ and $0<\eps<\alpha$ then
\begin{equation*}
L^{\alpha+\eps,p}(\R^N)\hookrightarrow W^{\alpha,p}(\R^N)\hookrightarrow L^{\alpha-\eps,p}(\R^N).
\end{equation*}
\end{Theorem}
Finally, given $s\in (0,1)$, define
\begin{equation*}
L^{s,p}_0(\Omega):=\{u\in L^{s,p}(\R^N): u=0\;\text{a.e. in}\;\R^N\setminus\Omega \}.
\end{equation*}
Thanks to Theorem \ref{besselspace} we infer
\begin{equation}\label{comparisonbound}
L^{s+\eps,p}_0(\Omega)\hookrightarrow W^{s,p}_0(\Omega)
\hookrightarrow L^{s-\eps,p}_0(\Omega)\quad\forall\,\eps\in (0,s).
\end{equation}
The next basic notion is taken from \cite{SS1}. For $0<\alpha<N$, let
\begin{equation*}
\gamma(N,\alpha):=\frac{\Gamma\left(\frac{N-\alpha}{2}\right)}{\pi^\frac{N}{2} 2^\alpha\Gamma\left(\frac{\alpha}{2}\right)},\quad 
I_\alpha(x):=\frac{\gamma(N,\alpha)}{|x|^{N-\alpha}},\quad x\in\R^N\setminus\{0\}.
\end{equation*}
If $u\in L^p(\R^N)$ and $I_{1-s}\ast u$ makes sense then the vector
$$D^s u:=\left(\frac{\partial}{\partial x_1}(I_{1-s}\ast u),\ldots,
\frac{\partial}{\partial x_N}(I_{1-s}\ast u)\right),$$
where partial derivatives are understood in a distributional sense, is called distributional Riesz $s$-fractional gradient of $u$. Theorem 1.2 in \cite{SS1} ensures that
$$D^s u=I_{1-s}\ast Du\quad\forall\, u\in C^\infty_c(\R^N).$$
Further, $D^s u$ looks like the natural extension of $\nabla u$ to the fractional framework, In fact, it exhibits analogous properties and, roughly speaking, $D^s u\to\nabla u$ when $s\to 1^-$; see, e.g., \cite[Section 2]{GMM-FCAA}.

According to \cite[Definition 1.5]{SS1}, $X^{s,p}(\R^N)$ denotes the completion of $C^\infty_c(\R^N)$ with respect to the norm
$$\|u\|_{X^{s,p}(\R^N)}:=\left(\|u\|_{L^p(\R^N)}^p+\|D^s u\|_{L^p(\R^N)}^p\right)^\frac{1}{p}.$$
Since, by \cite[Theorem 1.7]{SS1}, $X^{s,p}(\R^N)=L^{s,p}(\R^N)$ we can deduce
many facts about $X^{s,p}(\R^N)$ from the existing literature on $L^{s,p}(\R^N)$. Moreover, if
\begin{equation*}
X^{s,p}_{0}(\Omega):=\{u \in X^{s,p}(\R^N):u=0\;\text{a.e. in}\;\R^N\setminus\Omega\},
\end{equation*}
then $X^{s,p}_{0}(\Omega)=L^{s,p}_{0}(\Omega)$.  
\section{Existence results}
To shorten notation, for $i=1,2$, we set $U_i:=W^{s_i,p_i}_0(\Omega)$ and denote by $\langle\cdot,\cdot\rangle_i$ the duality brackets of $U_i$. Lemma 2.6 in \cite{BP} guarantees that
\begin{equation} \label{emb3}
U_i:=W^{s_i,p_i}_0(\Omega)\hookrightarrow W^{t_i,q_i}_0(\Omega).
\end{equation}
Hence, the  differential operator $u\mapsto (-\Delta)^{s_i}_{p_i} u+\mu_i
(-\Delta)^{t_i}_{q_i} u$ turns out well-defined on $U_i$. Let $A_{s_i,t_i}:U_i\to U^*_i$ be given by
\begin{equation*}
\begin{split}
\langle A_{s_i,t_i}(u),v\rangle_i
&:=\int_{\R^N\times\R^N}
\frac{|u(x)-u(y)|^{p_i-2}(u(x)-u(y))(v(x)-v(y))}{|x-y|^{N+p_i s_i}}\dx\,\dy\\
&+\mu_i\int_{\R^N\times\R^N}
\frac{|u(x)-u(y)|^{q_i-2}(u(x)-u(y))(v(x)-v(y))}{|x-y|^{N+q_i t_i}}\dx\,\dy
\end{split}
\end{equation*}
for every $u,v\in U_i$. Thanks to properties $({\rm p}_1)$--$({\rm p}_2)$ stated in Section \ref{prelim}, $A_{s_i,t_i}$ is bounded and continuous. Consequently, 
\begin{Lemma}\label{propA}
Under $({\rm H}_1)$, the operator $A:U_1\times U_2\to U^*_1\times U^*_2$  defined by
\begin{equation*} 
A(u_1,u_2):=(A_{s_1,t_1}(u_1),A_{s_2,t_2}(u_2))\;\;\forall\, (u_1,u_2)\in U_1\times U_2
\end{equation*}
maps bounded sets into bounded sets and is continuous.
\end{Lemma}
Next, put, provided $(u_1,u_2)\in U_1\times U_2$,
\begin{equation*}
\begin{split}\mathcal{S}_{F_1,F_2}(u_1,u_2):=\{ (w_1,w_2) &
\in L^{({p_1^*})'}(\Omega)\times L^{(p^*_2)'}(\Omega):\\
& w_i(\cdot)\in F_i(\cdot,u_1,u_2,D^{r_1} u_1,D^{r_2} u_2)\mbox{ a.e. in }\Omega,\; i=1,2\},  
\end{split}
\end{equation*}
with $p^*_i$ as in \eqref{defpstar}.
\begin{Lemma}\label{propS}
Let $({\rm H}_1)$--$({\rm H}_3)$ be satisfied. Then:
\begin{itemize}
\item[$({\rm a}_1)$] $\mathcal{S}_{F_1,F_2}(u_1,u_2)$ turns out nonempty, convex, closed for all $(u_1,u_2)\in U_1\times U_2$.
\item[$({\rm a}_2)$] The multifunction $\mathcal{S}_{F_1,F_2}:U_1\times U_2\to 2^{L^{({p_1^*})'}(\Omega)\times L^{(p^*_2)'}(\Omega)}$ is bounded and strongly-weakly upper semi-continuous.
\end{itemize}
\end{Lemma}
\begin{proof}
Since $r_i<s_i$, if $\eps\in (0,s_i-r_i)$, combining Proposition \ref{fracemb} with \eqref{comparisonbound} yields
\begin{equation*}
W^{s_i,p_i}_0(\Omega)\hookrightarrow W^{r_i+\eps,p_i}_0(\Omega)
\hookrightarrow L^{r_i,p_i}_0(\Omega).
\end{equation*}
Thus,
$$(u_1,u_2)\in U_1\times U_2\implies (D^{r_1}u_1,D^{r_2}u_2)\in L^{p_1}(\Omega)\times
L^{p_2}(\Omega).$$
Now, pick any $(u_1,u_2)\in U_1\times U_2$. Through $({\rm H_2})$ and Proposition \ref{supresult} we realize that $F_i(\cdot,u_1,u_2,D^{r_1}u_1,D^{r_2}u_2)$ admits a measurable selection $w_i:\Omega\to\R$. By $({\rm H}_3)$ one has
\begin{align*}
\| w_1\|_{(p_1^*)'}^{(p_1^*)'} 
& \le\int_\Omega\Big[ m_1\Big(|u_1|^{p_1^*-1}+|u_2|^\frac{p_2^*}{(p_1^*)'}
+|D^{r_1} u_1|^\frac{p_1}{(p_1^*)'}+|D^{r_2} u_2|^\frac{p_2}{(p_1^*)'}\Big)
+\delta_1\Big]^{(p_1^*)'}\,{\rm d}x\\
& \le (m_1+1)\Big(\|\delta_1\|_{(p_1^*)'}^{(p_1^*)'}
+\|u_1\|_{p_1^*}^{p_1^*}+\|u_2\|^{p_2^*}_{p_2^*}
+\|D^{r_1} u_1\|_{p_1}^{p_1}+\|D^{r_2} u_2\|_{p_2}^{p_2}\Big)<+\infty
\end{align*}
as well as $\| w_2\|_{(p_2^*)'}<+\infty$. Hence, $\mathcal{S}_{F_1,F_2}(u_1,u_2)\neq \emptyset$. This proves $({\rm a}_1)$, because convexity and closing follow at once from the analogous properties of $F_i$. 

Let us next verify $({\rm a}_2)$. The above inequalities also guarantee that $\mathcal{S}_{F_1,F_2}$ maps bounded sets into bounded sets. If $B$ is a nonempty weakly closed subset of $L^{(p_1^*)'}(\Omega)\times L^{(p_2^*)'}(\Omega)$ while $\{(u_{1,n},u_{2,n})\} \subseteq\mathcal{S}_{F_1,F_2}^-(B)$ converges to $(u_1,u_2)$ in $U_1\times U_2$, then $\{(u_{1,n},u_{2,n})\}\subseteq U_1\times U_2$ turns out bounded. The same holds true concerning the set
$$\bigcup_{n\in\N}\mathcal{S}_{F_1,F_2}(u_{1,n},u_{2,n})\subseteq L^{(p_1^*)'}(\Omega)\times L^{(p_2^*)'}(\Omega).$$
So, up to sub-sequences, there exists $(w_{1,n},w_{2,n})\in\mathcal{S}_{F_1,F_2}(u_{1,n},u_{2,n})\cap B$, $n\in\mathbb{N}$, such that
$$(w_{1,n},w_{2,n})\rightharpoonup(w_1,w_2)\quad\mbox{in}\quad
L^{(p_1^*)'}(\Omega)\times L^{(p_2^*)'}(\Omega).$$
One evidently has $(w_1,w_2)\in B$, because $B$ is weakly closed. Mazur's principle provides a sequence $\{(\Tilde{w}_{1,n},\Tilde{w}_{2,n})\}$ of convex combinations of $\{(w_{1,n}, w_{2,n})\}$ satisfying
$$(\Tilde{w}_{1,n},\Tilde{w}_{2,n})\to (w_1,w_2)\;\;\mbox{in}\;\;
L^{(p_1^*)'}(\Omega)\times L^{(p_2^*)'}(\Omega).$$
By $({\rm H_2})$, this easily entails
$$w_i(x)\in F_i(x,u_1(x),u_2(x),D^{r_1} u_1(x),D^{r_2} u_2(x))\;\;\mbox{for a.e.}\;\; x\in \Omega,\; i=1,2.$$
Consequently, $(w_1,w_2)\in\mathcal{S}_{F_1,F_2}(u_1,u_2)\cap B$, i.e., $(u_1,u_2)\in\mathcal{S}_{F_1,F_2}^-(B)$, as desired. 
\end{proof}
Our existence result can be established after introducing some suitable constants and the notion of generalized solution to \eqref{prob}. Since $r_i<s_i$, $i=1,2$, embeddings \eqref{comparisonbound} produce
\begin{equation}\label{gradest}
\Vert D^{r_1} u_1\Vert_{p_1}^{p_1}\leq\hat{c}_1\Vert u_1\Vert_{s_1,p_1}^{p_1}\;\;\forall\, u_1\in U_1,\quad
\Vert D^{r_2} u_2\Vert_{p_2}^{p_2}\leq\hat{c}_2\Vert u_2\Vert_{s_2,p_2}^{p_2}\;\;\forall\, u_2\in U_2,
\end{equation}
with appropriate $\hat{c}_i>0$. Via \eqref{emb3} and its analogue for couples $(s_2,p_2)$--$(t_2,q_2)$ we next have
\begin{equation}\label{tisi}
\Vert u_1\Vert_{t_1,q_1}^{q_1}\leq\Tilde{c}_1\Vert u_1\Vert_{s_1,p_1}^{p_1}\;\;\forall\, u_1\in U_1,\quad 
\Vert u_2\Vert_{t_2,q_2}^{q_2}\leq\Tilde{c}_2\Vert u_2\Vert_{s_2,p_2}^{p_2}\;\;\forall\, u_2\in U_2,
\end{equation}
where $\Tilde{c}_i>0$. Finally, set
$$\langle (u_1,u_2),(v_1,v_2)\rangle:=\langle u_1,v_1\rangle_1+\langle u_2,v_2\rangle_2,\quad (u_1,u_2)\in U_1\times U_2,\; (v_1,v_2)\in U^*_1\times U^*_2.$$
\begin{Definition}\label{Def1}
We say that $(u_1,u_2)\in U_1\times U_2$ is a generalized solution of \eqref{prob} if there exist two sequences 
$(u_{1,n},u_{2,n})\in U_1\times U_2$ and $(w_{1,n},w_{2,n})\in\mathcal{S}_{F_1,F_2}(u_{1,n},u_{2,n})$ fulfilling: 
\begin{itemize}
\item[{\rm(i)}] $(u_{1,n},u_{2,n})\rightharpoonup (u_1,u_2)$ in $U_1\times U_2$;
\item[{\rm(ii)}] $A(u_{1,n},u_{2,n})-(w_{1,n},w_{2,n})\rightharpoonup 0$ in $U_1^*\times U_2^*$;
\item[\rm(iii)] $\displaystyle{\lim_{n\to\infty}}\langle A(u_{1,n},u_{2,n})-(w_{1,n},w_{2,n}),(u_{1,n}-u_1,u_{2,n}-u_2)\rangle=0$.
\end{itemize} 
\end{Definition}
\begin{Theorem}\label{firstthm}
If $({\rm H_1})$--$({\rm H_4})$ are satisfied and, moreover,
\begin{equation}\label{ineq}
\frac{M_1+M_2}{\lambda_{1,p_i,s_i}}+\hat{c}_i(M_1'+M_2')+|\Tilde{c}_i|\mu_i|<1,\;\; i=1,2,   \end{equation}
then Problem \eqref{prob} admits a generalized solution.
\end{Theorem}
\begin{proof}
The space $U_1\times U_2$ is separable, therefore it possesses a Galerkin's basis, namely a sequence $\{E_n\}$ of linear sub-spaces of $U_1\times U_2$ such that:
\begin{itemize}
\item[$({\rm i}_1)$] ${\rm dim}(E_n)<\infty\;\;\forall\, n\in\mathbb{N}$;
\item[$({\rm i}_2)$] $E_n\subseteq E_{n+1}\;\;\forall\, n\in\mathbb{N}$;
\item[$({\rm i}_3)$] $\overline{\cup_{n=1}^{\infty}E_n}=U_1\times U_2$.
\end{itemize}
Pick any $n\in\mathbb{N}$. Consider the problem: Find $(u_1,u_2)\in E_n$ fulfilling
\begin{equation}\label{eqnss3.2}
A(u_1,u_2)-\mathcal{S}_{F_1,F_2}(u_1,u_2)\ni 0\;\;\mbox{in}\;\; E_n^*.
\end{equation}
By Lemma \ref{propS} the multifunction
$$(A-\mathcal{S}_{F_1,F_2})\lfloor_{E_n}:E_n\to 2^{E_n^*}$$
takes convex closed values, maps bounded sets into bounded sets, and is upper semi- continuous. If $(u_1,u_2)\in U_1\times U_2$ and $(w_1,w_2)\in\mathcal{S}_{F_1,F_2}(u_1,u_2)$ then, thanks to $({\rm H}_4)$,  we have
\begin{align*}
\langle A(u_1,u_2) & -(w_1,w_2),(u_1,u_2)\rangle\geq
\|u_1\|_{s_1,p_1}^{p_1}+\|u_2\|_{s_2,p_2}^{p_2}
-|\mu_1|\|u_1\|_{t_1,q_1}^{q_1}-|\mu_2|\|u_2\|_{t_2,q_2}^{q_2}\\
& - \int_\Omega\left[ M_1(|u_1|^{p_1}+|u_2|^{p_2})+M_1'(|D^{r_1} u_1|^{p_1}+|D^{r_2} u_2|^{p_2})+\sigma_1\right]\dx\\
& -\int_\Omega\left[ M_2(|u_1|^{p_1}+|u_2|^{p_2})+M_2'(|D^{r_1} u_1|^{p_1}+|D^{r_2} u_2|^{p_2})+\sigma_2\right]\dx.
\end{align*}
Using $({\rm p}_3)$ yields
\begin{align*}
\langle & A(u_1,u_2)-(w_1,w_2),(u_1,u_2)\rangle\\
& \geq\left(1-\frac{M_1+M_2}{\lambda_{1,p_1,s_1}}\right)\|u_1\|_{s_1,p_1}^{p_1}
+\left(1-\frac{M_1+M_2}{\lambda_{1,p_2,s_2}}\right)\| u_2\|_{s_2,p_2}^{p_2}
-|\mu_1|\| u_1\|_{t_1,q_1}^{q_1}-|\mu_2|\|u_2\|_{t_2,q_2}^{q_2}\\
& \phantom{ppp}-\int_\Omega (M_1'+M_2')|D^{r_1} u_1|^{p_1}\dx
-\int_\Omega (M_1'+M_2')|D^{r_2} u_2|^{p_2}\dx-\|\sigma_1\|_1- \|\sigma_2\|_1, 
\end{align*}
whence, on account of \eqref{gradest},
\begin{align*}
\langle & A(u_1,u_2)-(w_1,w_2),(u_1,u_2)\rangle\\
&\geq\left[1-\frac{M_1+M_2}{\lambda_{1,p_1,s_1}}-\hat{c}_1(M_1'+M_2')\right] \|u_1\|_{s_1,p_1}^{p_1}
+\left[1-\frac{M_1+M_2}{\lambda_{1,p_2,s_2}}-\hat{c}_2(M_1'+M_2')\right]
\|u_2\|_{s_2,p_2}^{p_2}\\
& \phantom{ppp}-|\mu_1|\| u_1\|_{t_1,q_1}^{q_1}-|\mu_2|\|u_2\|_{t_2,q_2}^{q_2}
-\Vert\sigma_1\Vert_1-\Vert\sigma_2\Vert_1.
\end{align*}
Finally, through \eqref{tisi} we obtain
\begin{align*}
& \langle A(u_1,u_2)-(w_1,w_2),(u_1,u_2)\rangle\\
& \geq\left[ 1-\frac{M_1+M_2}{\lambda_{1,p_1,s_1}}-\hat{c}_1(M_1'+M_2')
-|\mu_1|\Tilde{c}_1\right] \|u_1\|_{s_1,p_1}^{p_1}\\
& +\left[1-\frac{M_1+M_2}{\lambda_{1,p_2,s_2}}-\hat{c}_2(M_1'+M_2')
-|\mu_2|\Tilde{c}_2\right]\|u_2\|_{s_2,p_2}^{p_2}
-\Vert\sigma_1\Vert_1-\Vert\sigma_2\Vert_1\\
& \geq\min_{i=1,2}\left[1-\frac{M_1+M_2}{\lambda_{1,p_i,s_i}}-\hat{c}_i(M_1'+M_2')
-|\mu_i|\Tilde{c}_i\right] \left(\|u_1\|_{s_1,p_1}^{p_1}+\|u_2\|_{s_2,p_2}^{p_2}\right)
-\Vert\sigma_1\Vert_1-\Vert\sigma_2\Vert_1,
\end{align*}
namely
\begin{equation}\label{coerc}
\langle A(u_1,u_2)-(w_1,w_2),(u_1,u_2)\rangle\geq\alpha \left(\|u_1\|_{s_1,p_1}^{p_1}+\|u_2\|_{s_2,p_2}^{p_2}\right)-\beta,    
\end{equation}
where
$$\alpha:=\min_{i=1,2}\left[1-\frac{M_1+M_2}{\lambda_{1,p_i,s_i}}-\hat{c}_i(M_1'+M_2')
-|\mu_i|\Tilde{c}_i\right], \quad\beta:=\Vert\sigma_1\Vert_1+\Vert\sigma_2\Vert_1.$$
Since \eqref{ineq} holds, the multifunction $(A-\mathcal{S}_{F_1,F_2})\lfloor_{E_n}$ turns out coercive. Now, Theorem~\ref{surthm} can be applied, and there exists a solution $(u_{1,n},u_{2,n})\in E_n$ to Problem \eqref{eqnss3.2}, i.e.,
\begin{equation}\label{eqnss3.4}
A(u_{1,n},u_{2,n})-(w_{1,n},w_{2,n})= 0\;\;\mbox{in}\;\; E^*_n 
\end{equation}
for suitable $(w_{1,n},w_{2,n})\in\mathcal{S}_{F_1,F_2}(u_{1,n},u_{2,n})$. From \eqref{coerc}, written with $(u_1,u_2):=(u_{1,n},u_{2,n})$, and \eqref{eqnss3.4} it follows
$$0\geq\alpha\left(\|u_{1,n}\|_{s_1,p_1}^{p_1}+\| u_{2,n}\|_{s_2,p_2}^{p_2}\right)-\beta\quad\forall\, n\in\mathbb{N}.$$
Thus, $\{(u_{1,n},u_{2,n})\}\subseteq U_1\times U_2$ is bounded. By reflexivity one has $(u_{1,n},u_{2,n})\rightharpoonup (u_1,u_2)$ in $U_1\times U_2$, taking a sub-sequence when necessary. Consequently, (i) of Definition \ref{Def1} holds. Through Lemma 6 and \eqref{eqnss3.4} we next infer that $\{(w_{1,n},w_{2,n})\}\subseteq L^{(p_1^*)'}(\Omega)\times L^{(p_3^*)'}(\Omega)$ turns out bounded. Therefore, always up to sub-sequences,
\begin{equation}\label{weaklim}
A(u_{1,n},u_{2,n})-(w_{1,n},w_{2,n})\rightharpoonup (\varphi_1,\varphi_2)\;\;\mbox{in} \;\; U_1^*\times U_2^*.
\end{equation}
Now, given any $(v_1,v_2)\in \cup_{n=1}^\infty E_n$, Property $({\rm i}_2)$ and \eqref{eqnss3.4} yield
$$\langle (\varphi_1,\varphi_2),(v_1,v_2)\rangle
=\lim_{n\to\infty}\langle A(u_{1,n},u_{2,n})-(w_{1,n},w_{2,n}),(v_1,v_2)\rangle=0.$$
Because of $({\rm i}_3)$ this forces 
\begin{equation}\label{zerosol}
(\varphi_1,\varphi_2)= 0\;\;\mbox{in}\;\; U_1^*\times U_2^*,    
\end{equation}
namely, condition (ii) is true. Using \eqref{eqnss3.4}--\eqref{zerosol} entails
\begin{equation}\label{lateruse}
\begin{split}
\langle A(u_{1,n},u_{2,n})-(w_{1,n},w_{2,n}), & (u_{1,n}-u_1,u_{2,n}-u_2)\rangle\\
& =-\langle A(u_{1,n},u_{2,n})-(w_{1,n},w_{2,n}),(u_1,u_2)\rangle\to 0
\end{split}
\end{equation}
as $n\to\infty$, which shows (iii) in Definition \ref{Def1}. Summing up, the pair $(u_1,u_2)$ turns out a generalized solution to \eqref{prob}.
\end{proof}
If we strengthen $({\rm H}_3)$ as follows:
\begin{itemize}
\item[$({\rm H}_3)'$] For each $i=1,2$ there exist $\rho_i,\sigma_i\in (1,p^*_i)$, $m_i>0$, and $\delta_i\in L^{\sigma_i'}(\Omega)$ such that
\begin{align*}
|F_i(x,y_1,y_2,z_1,z_2)|
\le & m_i \left(|y_1|^\frac{p^*_1}{\rho_i'}+|y_2|^\frac{p_2^*}{\rho_i'}
+|z_1|^\frac{p_1}{\rho_i'}+|z_2|^\frac{p_2}{\rho_i'}\right)+\delta_i(x)
\end{align*}
a.e. in $\Omega$ and for all $(y_1,y_2,z_1,z_2)\in\R^2\times\R^{2N}$,
\end{itemize}
then the next notion of strongly generalized solution can be given. Obviously, $({\rm H}_3)'$ implies $({\rm H}_3)$, because $\rho_i<p^*_i$ forces
\begin{equation*}
\frac{\kappa}{\rho'_i}<\frac{\kappa}{(p^*_i)'}\quad\forall\, \kappa\in
\{ p_1,p_2,p^*_1,p^*_2\}. 
\end{equation*}
\begin{Definition}\label{Def2}
We say that $(u_1,u_2)\in U_1\times U_2$ is a strongly generalized solution to \eqref{prob} if there exist two sequences 
$(u_{1,n},u_{2,n})\in U_1\times U_2$ and $(w_{1,n},w_{2,n})\in\mathcal{S}_{F_1,F_2}(u_{1,n},u_{2,n})$ satisfying (i) and (ii) of Definition \ref{Def1} and, moreover,
\begin{itemize}
\item[${\rm(iii)}'$] $\displaystyle{\lim_{n\to\infty}}\langle A(u_{1,n},u_{2,n}),(u_{1,n}-u_1,u_{2,n}-u_2)\rangle=0$.
\end{itemize} 
\end{Definition}
\begin{Theorem}\label{thmtwo}
Under assumptions $({\rm H}_1)$--$({\rm H}_2)$, $({\rm H}_3)'$, $({\rm H}_4)$, and \eqref{ineq}, Problem \eqref{prob} admits a strongly generalized solution.
\end{Theorem}
\begin{proof}
Reasoning as in the proof of Theorem \ref{firstthm}  yields both $(u_1,u_2)\in U_1\times U_2$ and two sequences $(u_{1,n},u_{2,n})\in U_1\times U_2$, $(w_{1,n},w_{2,n})\in \mathcal{S}_{F_1,F_2}(u_{1,n},u_{2,n})$ that comply with (i)--(ii) in Definition \ref{Def1} as well as \eqref{lateruse}. Thus, it remains to show ${\rm(iii)}'$. By $({\rm H}_3)'$ and H\"older's inequality we have
\begin{equation*}
\begin{split}
& \left|\int_\Omega w_{i,n}(u_{i,n}-u_i)\dx\right|\\
&\le m_i\int_\Omega\left(|u_{1,n}|^{\frac{p_1^*}{\rho_i'}} +|u_{2,n}|^{\frac{p_2^*}{\rho_i'}}+|\nabla u_{1,n}|^{\frac{p_1}{\rho_i'}}
+|\nabla u_{2,n}|^{\frac{p_2}{\rho_i'}}\right)|u_{i,n}-u_i|\dx
+\int_\Omega\delta_i|u_{i,n}-u_i|\dx\\
&\le m_i\left(\|u_{1,n}\|^{\rho_i'}_{p_1^*}+\|u_{2,n}\|^{\rho_i'}_{p_2^*}
+\|u_{1,n}\|^{\rho_i'}_{1,p_1}+\|u_{2,n}\|^{\rho_i'}_{1,p_2}\right)
\|u_{i,n}-u_i\|_{\rho_i}+\|\delta_i\|_{\sigma_i'}\|u_{i.n}-u_i\|_{\sigma_i}\\
&\le C\|u_{i,n}-u_i\|_{\rho_i}+\|\delta_i\|_{\sigma_i'}\|u_{i,n}-u_i\|_{\sigma_i} \quad\forall\, n\in\N,
\end{split}	
\end{equation*}
because $\{u_{i,n}\}\subseteq U_i$ turns out bounded. The condition $\rho_i\vee\sigma_i<p_i^*$, then, forces $u_{i,n}\to u_i$ in $L^{\rho_i}(\Omega)\cap 
L^{\sigma_i}(\Omega)$, where a sub-sequence is considered if necessary; see Proposition \ref{fracemb}. Hence,
\begin{equation}\label{supop}
\lim_{n\to\infty}\int_\Omega w_{i,n}(u_{i,n}-u_i)\dx=0,\quad i=1,2.  
\end{equation}
Through \eqref{lateruse}--\eqref{supop}, we arrive at
\begin{equation*}
\lim_{n\to\infty}\langle A(u_{1,n},u_{2,n}),(u_{1,n}-u_1,u_{2,n}-u_2)\rangle=0,
\end{equation*}
namely $({\rm iii})'$ of Definition \ref{Def2} also holds.
\end{proof}
Finally, recall that $(u_1,u_2)\in U_1\times U_2$ is called a \textit{weak solution} to \eqref{prob} when there exists $(w_1,w_2)\in\mathcal{S}_{F_1,F_2}(u_1,u_2)$ such that
\begin{equation}\label{wsol}
A(u_1,u_2)=(w_1,w_2)\;\;\mbox{in}\;\; U_1^*\times U_2^*.    
\end{equation}
\begin{Corollary}
Let the hypotheses of Theorem \ref{thmtwo} be satisfied and let $\mu_1\wedge\mu_2\geq 0$. Then Problem \eqref{prob} possesses a weak solution. 
\end{Corollary}
\begin{proof}
Keep the same notation of the previous proof. Since $\mu_i\geq 0$, gathering $({\rm p}_1)$ with Proposition \ref{sumop} together ensures that $A_{s_i,t_i}$ is of type $({\rm S})_+$.
Therefore, from $({\rm iii})'$ it follows $(u_{1,n},u_{2,n})\to (u_1,u_2)$ in $U_1\times U_2$. On the other hand, $({\rm a}_2)$ in Lemma \ref{propS} produces, up to subsequences, $(w_{1,n},w_{2,n})\rightharpoonup (w_1,w_2)$ in $U_1^*\times U_2^*$. Now, through $({\rm ii})$ and Lemma \ref{propA} we easily arrive at \eqref{wsol}.
\end{proof}
\section*{Acknowledgment}
This work was supported in part by: 1) the Research project of MIUR Prin 2022 {\it Nonlinear differential problems with applications to real phenomena} (Grant No. 2022ZXZTN2); 2) the Natural Science Foundation of Guangxi (Grant Nos. 2021GXNSFFA196004 and 2024GXNSFBA010337); 3) the National Natural Science Foundation of China (Grant No. 12371312); 4) the Natural Science Foundation of Chongqing (Grant No. CSTB2024 NSCQ-JQX0033); 5) the  Postdoctoral Fellowship Program of CPSF (Grant No. GZC20241534); 6) the Systematic Project of Center for Applied Mathematics of Guangxi (Yulin Normal University) (Grant Nos. 2023CAM002 and 2023CAM003).

The second author is a member of the GNAMPA of INdAM.

\end{document}